\documentclass[ECP]{ejpecp}

\usepackage[T1]{fontenc}

\usepackage{amsthm,amsmath,amssymb,graphicx,mathtools}
\usepackage{float} 

\SHORTTITLE{Randomized Riemann zeta function}

\TITLE{\vspace{-35mm}\includegraphics[width=150mm]{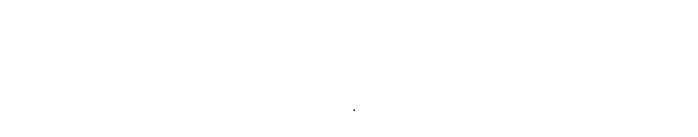} Poisson-Dirichlet statistics for the extremes\\[2mm]of a randomized Riemann zeta function\thanks{F. Ouimet is supported by a NSERC Doctoral Program Alexander Graham Bell scholarship (CGS D3).}}

\AUTHORS{%
  Fr\'ed\'eric Ouimet\footnote{Universit\'e de Montr\'eal, Canada.
    \EMAIL{ouimetfr@dms.umontreal.ca}}}


\AMSSUBJ{11M06; 60F05; 60G60; 60G70}

\SUBMITTED{February 18, 2018}
\ACCEPTED{July 18, 2018}

\VOLUME{0}
\YEAR{2018}
\PAPERNUM{0}
\DOI{10.1214/YY-TN}

\ABSTRACT{In \cite{arXiv:1706.08462}, the authors prove the convergence of the two-overlap distribution at low temperature for a randomized Riemann zeta function on the critical line.
We extend their results to prove the Ghirlanda-Guerra identities.
As a consequence, we find the joint law of the overlaps under the limiting mean Gibbs measure in terms of Poisson-Dirichlet variables.
It is expected that we can adapt the approach to prove the same result for the Riemann zeta function itself.}

\newcommand{\N}{\mathbb{N}}

\newcommand{\R}{\mathbb{R}}
\newcommand{\C}{\mathbb{C}}

\newcommand{\EE}{\mathbb{E}}
\newcommand{\bb}[1]{\boldsymbol{#1}}

\begin{document}

\section{Introduction}\label{sec:intro}

    Following recent conjectures of \cite{FyodorovHiaryKeating2012} and \cite{MR3151088} about the limiting law of the {\it Gibbs measure} and the limiting law of the maximum for the Riemann zeta function on bounded random intervals of the critical line, progress have been made in the mathematics literature.
    If $\tau$ is sampled uniformly in $[T,2T]$ for some large $T$, then it is expected that the limiting law of the Gibbs measure (see \eqref{def:gibbs.measure}) at low temperature for the field $(\log |\zeta(\frac{1}{2} + i(\tau + h))|, h\in [0,1])$ is a one-level Ruelle probability cascade (see e.g.\hspace{-0.3mm} \cite{MR875300}) and the law of the maximum is asymptotic to $\log \log T - \frac{3}{4} \log \log \log T + \mathcal{M}_T$ where $(\mathcal{M}_T, T\geq 2)$ is a sequence of random variables converging in distribution.
    For a randomized version of the Riemann zeta function (see \eqref{def:X}), the first order of the maximum was proved in \cite{arXiv:1304.0677}, the second order of the maximum was proved in \cite{MR3619786}, and the limiting two-overlap distribution was found in \cite{arXiv:1706.08462} (see Theorem \ref{thm:limiting.two.overlap.distribution} below). The tightness of the recentered maximum is still open (see \cite{arXiv:1807.04860}).
    In this short paper, we complete the analysis of \cite{arXiv:1706.08462} by proving the Ghirlanda-Guerra (GG) identities in the limit $T\to \infty$ (see Theorem \ref{thm:extended.GG.identities}).
    As is well known in the spin glass literature (see e.g.\hspace{-0.3mm} Chapter 2 in \cite{MR3052333}), the limiting law of the two-overlap distribution, with a finite support, together with the GG identities allow a complete description of the limiting law of the Gibbs measure as a {\it Ruelle probability cascade} with finitely many levels (a random measure with a tree structure and Poisson-Dirichlet weights at each level).
    Our main result (Theorem \ref{thm:Poisson.Dirichlet}) describes the joint law of the overlaps under the limiting mean Gibbs measure in terms of Poisson-Dirichlet weights.
    It is expected that the approach presented here, which mostly stems from the work of \cite{MR3211001}, \cite{MR2070334} and \cite{MR3052333} on other models, can be adapted to prove the same result for the (true) Riemann zeta function on bounded random intervals of the critical line. At present, for the (true) Riemann zeta function, the first order of the maximum is proved conditionally on the Riemann hypothesis in \cite{doi:10.1007/s00440-017-0812-y} and unconditionally in \cite{arXiv:1612.08575}.

    The paper is organised as follows.
    In Section \ref{sec:definitions}, we give a few definitions.
    In Section \ref{sec:main.result}, the main result is stated and shown to be a consequence of the GG identities and the main result from \cite{arXiv:1706.08462} about the limiting two-overlap distribution.
    In Section \ref{sec:known.results}, we state known results from \cite{arXiv:1706.08462} that we will use to prove the GG identities.
    The GG identities are proven in Section \ref{sec:proof.GG} along with other preliminary results, see the structure of the proof in Figure \ref{fig:proof.structure}.
    For an explanation of the consequences of the GG identities and their conjectured universality for mean field spin glass models, we refer the reader to \cite{MR3628881}, \cite{MR3052333} and \cite{MR3024566}.

\section{Some definitions}\label{sec:definitions}

    Let $(U_p, p ~\text{primes})$ be an i.i.d.\hspace{-0.3mm} sequence of uniform random variables on the unit circle in $\C$.
    The random field of interest is
    \begin{equation}\label{def:X}
        X_h \circeq \sum_{p \leq T} W_p(h) \circeq \sum_{p \leq T} \frac{\text{Re}(U_p \, p^{-i h})}{p^{1/2}}, \quad h\in [0,1].
    \end{equation}
    This is a good model for the large values of $(\log |\zeta(\frac{1}{2} + i(\tau + h))|, h\in [0,1])$ for the following reason.
    Proposition 1 in \cite{arXiv:1304.0677} proves that, assuming the Riemann hypothesis, and for $T$ large enough, there exists a set $B\subseteq [T,T+1]$, of Lebesgue measure at least $0.99$, such that
    \begin{equation}
        \log |\zeta(\frac{1}{2} + i t)| = \text{Re}\left(\sum_{p \leq T} \frac{1}{p^{1/2 + it}} \frac{\log(T / p)}{\log T}\right) + O(1), \quad t\in B.
    \end{equation}
    If we ignore the smoothing term $\log(T / p) / \log T$ and note that the process $(p^{-i\tau}\hspace{-1mm}, p ~\text{primes})$, where $\tau$ is sampled uniformly in $[T,2T]$, converges (in the sense of convergence of its finite-dimensional distributions), as $T\to\infty$, to a sequence of independent random variables distributed uniformly on the unit circle (by computing the moments), then the model \eqref{def:X} follows.
    For more information, see Section 1.1 in \cite{MR3619786}.

    For simplicity, the dependence in $T$ will be implicit everywhere for $X$. Summations over $p$'s and $q$'s always mean that we sum over primes.
    For $\alpha\in [0,1]$, we denote truncated sums of $X$ as follows :
    \begin{equation}\label{def:X.alpha}
        X_h(\alpha) \circeq \sum_{p \leq \exp((\log T)^{\alpha})} W_p(h), \quad h\in [0,1],
    \end{equation}
    where $\sum_{\emptyset} \circeq 0$. Define the {\it overlap} between two points of the field by
    \begin{equation}\label{eq:correlation}
        \rho(h,h') \circeq \frac{\EE[X_h X_{h'}]}{\sqrt{\EE[X_h^2] \EE[X_{h'}^2]}}, \quad h,h'\in [0,1].
    \end{equation}
    For any $\alpha\in [0,1]$ and any $\beta > 0$, define the {\it (normalized) free energy of the perturbed model} by
    \begin{equation}\label{eq:free.energy}
        f_{\alpha,\beta,T}(u) \circeq \frac{1}{\log \log T} \log \int_0^1 e^{\beta (u X_h(\alpha) + X_h)} dh, \quad u > -1.
    \end{equation}
    The parameter $u$ is there to allow perturbations in the correlation structure of the model.
    When $u = 0$, we recover the {\it free energy}.
    Finally, for any Borel set $A\in \mathcal{B}([0,1])$, define the {\it Gibbs measure} by
    \begin{equation}\label{def:gibbs.measure}
        G_{\beta,T}(A) = \int_A \frac{e^{\beta X_h}}{\int_{[0,1]} e^{\beta X_{h'}} dh'} dh.
    \end{equation}
    The parameter $\beta$ is called the {\it inverse temperature} in statistical mechanics.

    \section{Main result}\label{sec:main.result}

    The main result of this article is to present a complete description of the joint law of the overlaps for the model \eqref{def:X}, under the {\it limiting mean Gibbs measure}
    \begin{equation}\label{eq:limiting.mean.Gibbs.measure}
        \lim_{T\to \infty} \EE G_{\beta,T}.
    \end{equation}
    We will show that, when $\beta > \beta_c \circeq 2$, this measure is the expectation $E$ of a random measure $\mu_{\beta}$ sampling orthonormal vectors in an infinite-dimensional separable Hilbert space, where the probability weights follow a
    \begin{equation*}
        \text{Poisson-Dirichlet distribution of parameter $\beta_c / \beta$.}
    \end{equation*}

    This is done through what is called the {\it Ghirlanda-Guerra identities}.
    These identities first appeared in \cite{MR1662161} and, 15 years later, it was proved in a celebrated work of Panchenko \cite{MR2999044} (a simple proof is given in \cite{MR2825947} when $E \mu_{\beta}$ has a finite support) that if a random measure on the unit ball of a separable Hilbert space satisfies an extended version of the Ghirlanda-Guerra identities, then we must have ultrametricity (a tree-like structure) of the overlaps under the mean of this random measure. This was an important step because it was well-known following the publication of \cite{MR1662161} that the Ghirlanda-Guerra identities and ultrametricity together completely determine the joint law of the overlaps, up to the distribution of one overlap.
    See e.g., Theorem 6.1 in \cite{BaffioniRosati2000}, Section 1.2 in \cite{MR1993891} (in the context of the REM model from \cite{MR575260}) and Theorem 1.13 in \cite{MR2070334} (in the context of the GREM model from \cite{Derrida1985}).

    Thus, from the work of Panchenko, proving the (extended) Ghirlanda-Guerra identities under \eqref{eq:limiting.mean.Gibbs.measure} implies ultrametricity and, consequently, determines the joint law of the overlaps, up to
    the {\it limiting two-overlap distribution}
    \begin{equation}\label{eq:limiting.two.overlap.distribution}
        \lim_{T\to\infty} \EE G_{\beta,T}[\bb{1}_{\{\rho(h,h') \in \, \cdot \, \}}],
    \end{equation}
    which \cite{arXiv:1706.08462} already determined for the model \eqref{def:X}.

    \begin{theorem}[Theorem 1 in \cite{arXiv:1706.08462}]\label{thm:limiting.two.overlap.distribution}
        For any $\beta > \beta_c \circeq 2$ and any Borel set $A\in \mathcal{B}([0,1])$,
        \begin{equation}\label{eq:thm:limiting.two.overlap.distribution.eq}
            \lim_{T\to \infty} \EE G_{\beta,T}^{\times 2} \big[\bb{1}_{\{\rho(h,h') \in A\}}\big] = \frac{2}{\beta} \bb{1}_A(0) + \left(1 - \frac{2}{\beta}\right) \bb{1}_A(1).
        \end{equation}
    \end{theorem}

    \begin{remark}
        The limiting two-overlap distribution in \eqref{eq:thm:limiting.two.overlap.distribution.eq} can be interpreted as a measure of relative distance between the extremes of the model.
    \end{remark}

    To state our main result, recall the definition of a Poisson-Dirichlet variable.
    For $0 < \theta < 1$, let $\eta = (\eta_i)_{i\in \N^*}$ be the atoms of a Poisson random measure on $(0,\infty)$ with intensity measure $\theta x^{-\theta - 1} dx$.
    A {\it Poisson-Dirichlet variable} $\xi$ of parameter $\theta$ is a random variable on the space of decreasing weights
    \begin{equation}\label{eq:space.decreasing.weights}
        \left\{(x_1,x_2,\ldots)\in [0,1]^{\N^*} :
        \begin{array}{l}
            1 \geq x_1 \geq x_2 \geq \ldots \geq 0  \\[1mm]
            \text{and}~ \sum_{i=1}^{\infty} x_i = 1
        \end{array}
        \right\}
    \end{equation}
    which has the same law as
    \begin{equation}
        \xi \stackrel{\text{law}}{=} \left(\frac{\eta_i}{\sum_{j=1}^{\infty} \eta_j}, ~i\in \N^*\right)_{\downarrow},
    \end{equation}
    where $\downarrow$ stands for the decreasing rearrangement.

    Here is the main result.

    \begin{theorem}[Main result]\label{thm:Poisson.Dirichlet}
        Let $\beta > \beta_c \circeq 2$ and let $\xi = (\xi_k)_{k\in \N^*}$ be a Poisson-Dirichlet variable of parameter $\beta_c / \beta$.
        Denote by $E$ the expectation with respect to $\xi$.
        For any continuous function $\phi : [0,1]^{s(s-1)/2} \to \R$ of the overlaps of $s$ points,
        \begin{equation}\label{eq:thm:Poisson.Dirichlet.eq}
            \begin{aligned}
            &\lim_{T\to \infty} \EE G_{\beta,T}^{\times s} \Big[\phi\Big(\big(\rho(h_l,h_{l'})\big)_{1 \leq l,l' \leq s}\Big)\Big] \\
            &\hspace{20mm}= E\left[\sum_{k_1,\ldots,k_s\in \N} \xi_{k_1} \cdots \xi_{k_s} \phi\Big(\big(\bb{1}_{\{k_l = k_{l'}\}}\big)_{1 \leq l,l' \leq s}\Big)\right].
            \end{aligned}
        \end{equation}
    \end{theorem}

    \begin{remark}
        The domain of $\phi$ is $[0,1]^{s(s-1)/2}$ here because the matrix $(\rho(h_l,h_{l'}))_{1 \leq l,l' \leq s}$ is symmetric and has $1$'s on the diagonal.
    \end{remark}

    \begin{remark}
        The proof of Theorem \ref{thm:Poisson.Dirichlet} is given in Section \ref{sec:proof.min.result}.
        As mentioned earlier, it is a consequence of Theorem \ref{thm:limiting.two.overlap.distribution}, Theorem \ref{thm:extended.GG.identities} and the ultrametric structure of the overlaps under the limiting mean Gibbs measure.
        To prove the extended Ghirlanda-Guerra identities in Section \ref{sec:proof.GG}, we will use the strategy developed in \cite{MR2070334,MR2070335} and used in \cite{MR3211001} and \cite{arXiv:1706.08462} (see Remark \ref{rem:beta.cases.explanation}).
        For an alternative strategy (which requires a stronger control on the path of the maximal particle in the tree structure), see \cite{MR3539644}.
    \end{remark}

    \begin{remark}\label{rem:beta.cases.explanation}
        In this paper, we state most of our results above the critical inverse temperature (i.e.\hspace{-0.3mm} at low temperature), namely when $\beta > \beta_c \circeq 2$, because that's the only interesting case. The description of the joint law of the overlaps under the limiting mean Gibbs measure turns out to be trivial when $\beta < \beta_c$.
        Here's why.

        When $\beta > \beta_c$, the Gibbs measure gives a lot of weight to the ``particles'' $h$ that are near the maximum's height in the tree structure underlying the model \eqref{def:X}. The result of Theorem \ref{thm:limiting.two.overlap.distribution} simply says that if you sample two particles under the Gibbs measure, then, in the limit and on average, either the particles branched off ``at the last moment'' in the tree structure (there are clusters of points reaching near the level of the maximum) or they branched off in the beginning. They cannot branch at intermediate scales.

        When $\beta < \beta_c$, the weights in the Gibbs measure are more spread out so that most contributions to the free energy actually come from particles reaching heights that are well below the level of the maximum in the tree structure. Hence, when two particles are selected from this larger pool of contributors that are not clustering, it can be shown that, in the limit and on average, the particles necessarily branched off in the beginning of the tree.
        The proof would follow the exact same strategy used in \cite{arXiv:1706.08462} :
        \begin{itemize}
            \item find the free energy of the perturbed model as a function of the perturbation parameter $u$,
            \item link the expectation of the derivative of the perturbed free energy at $u = 0$ with the two-overlap distribution by using an approximate integration by parts argument and the convexity of the free energy.
        \end{itemize}
        (We refer to this strategy as the {\it Bovier-Kurkova technique} since it is adapted from the strategy introduced in \cite{MR2070334,MR2070335} for the GREM model.)
        The computations would actually be easier in this case. One would find that
        \begin{equation}\label{eq:limiting.gibbs.measure.under.2}
            \lim_{T\to \infty} \EE G_{\beta,T}^{\times 2} \big[\bb{1}_{\{\rho(h,h') \in A\}}\big] = \bb{1}_A(0).
        \end{equation}
        In other words, when $\beta < \beta_c$, the limiting mean Gibbs measure only samples points that are uncorrelated (and thus far from each other) in the limiting tree structure.
        More generally, our main result (Theorem \ref{thm:Poisson.Dirichlet}), which describes the joint law of the overlaps under the limiting mean Gibbs measure, would say that for any continuous function $\phi : [0,1]^{s^2} \to \R$ of the overlaps of $s$ points,
        \begin{equation}\label{eq:limiting.joint.law.overlaps.under.2}
            \lim_{T\to \infty} \EE G_{\beta,T}^{\times s} \Big[\phi\Big(\big(\rho(h_l,h_{l'})\big)_{1 \leq l,l' \leq s}\Big)\Big] = \phi(I_s),
        \end{equation}
        where $I_s$ denotes the identity matrix of order $s$.
        In the critical case $\beta = \beta_c$, we obtain \eqref{eq:limiting.gibbs.measure.under.2} and \eqref{eq:limiting.joint.law.overlaps.under.2} with the same techniques.
    \end{remark}

    \section{Known results}\label{sec:known.results}

    In this section, we gather the results from \cite{arXiv:1706.08462} that we will use in Section \ref{sec:proof.GG} to prove the extended Ghirlanda-Guerra identities.
    The two propositions below are known convergence results for $f_{\alpha,\beta,T}$ and its derivative (with respect to $u$).
    We slightly reformulate them for later use.

    \begin{proposition}[Proposition 3 in \cite{arXiv:1706.08462}]\label{prop:mean.convergence.derivative.free.energy}
        Let $\beta > \beta_c \circeq 2$ and $0 < \alpha < 1$.
        Then,
        \begin{equation}\label{eq:prop.3.Arguin.Tai.2017}
            \frac{2}{\beta^2} \cdot \EE\big[f_{\alpha,\beta,T}'(0)\big] = \int_0^{\alpha} \EE G_{\beta,T}^{\times 2} [\bb{1}_{\{\rho(h,h') \leq y\}}] dy + o_T(1).
        \end{equation}
        Since $f_{\alpha,\beta,T}'(0) = \beta (\log \log T)^{-1} G_{\beta,T}[X_h(\alpha)]$, we can also write \eqref{eq:prop.3.Arguin.Tai.2017} as
        \begin{equation}\label{eq:prop.3.Arguin.Tai.2017.rewrite}
            \frac{1}{\beta} \cdot \frac{\EE G_{\beta,T}[X_h(\alpha)]}{\frac{1}{2} \log \log T} = \alpha - \EE G_{\beta,T}^{\times 2} [\int_0^{\alpha} \bb{1}_{\{y < \rho(h,h')\}} dy] + o_T(1).
        \end{equation}
    \end{proposition}

    \begin{proposition}[Equation 13, Proposition 4 and Lemma 14 in \cite{arXiv:1706.08462}]\label{prop:convergence.free.energy}
        Let $\beta > \beta_c \circeq 2$, $0 \leq \alpha \leq 1$ and $u > -1$.
        Then,
        \begin{equation}\label{eq:prop.4.Arguin.Tai.2017}
            \lim_{T\to \infty} f_{\alpha,\beta,T}(u) = f_{\alpha,\beta}(u) \circeq \left\{\hspace{-1mm}
                \begin{array}{ll}
                    \frac{\beta^2}{4} V_{\alpha,u}, &\mbox{if } u < 0, ~2 < \beta \leq 2 / \sqrt{V_{\alpha,u}}, \\[0.5mm]
                    \beta \sqrt{V_{\alpha,u}} - 1, &\mbox{if } u < 0, ~\beta > 2 / \sqrt{V_{\alpha,u}}, \\[0.5mm]
                    \beta(\alpha u + 1) - 1, &\mbox{if } u \geq 0, ~\beta > 2,
                \end{array}
            \right.
        \end{equation}
        where the limit holds in $L^1$, and where $V_{\alpha,u} \circeq (1+u)^2\alpha + (1-\alpha)$.
    \end{proposition}

\section{Proof of the extended Ghirlanda-Guerra identities}\label{sec:proof.GG}

    This section is dedicated to the proof of the extended Ghirlanda-Guerra identities (Theorem \ref{thm:extended.GG.identities}).
    We adopt a ``bottom-up'' style of presentation, where Theorem \ref{thm:extended.GG.identities} is the end goal.
    Here is the structure of the proof :


    \vspace{2mm}
    \begin{figure}[H] 
        \centering
        \includegraphics[width=11cm]{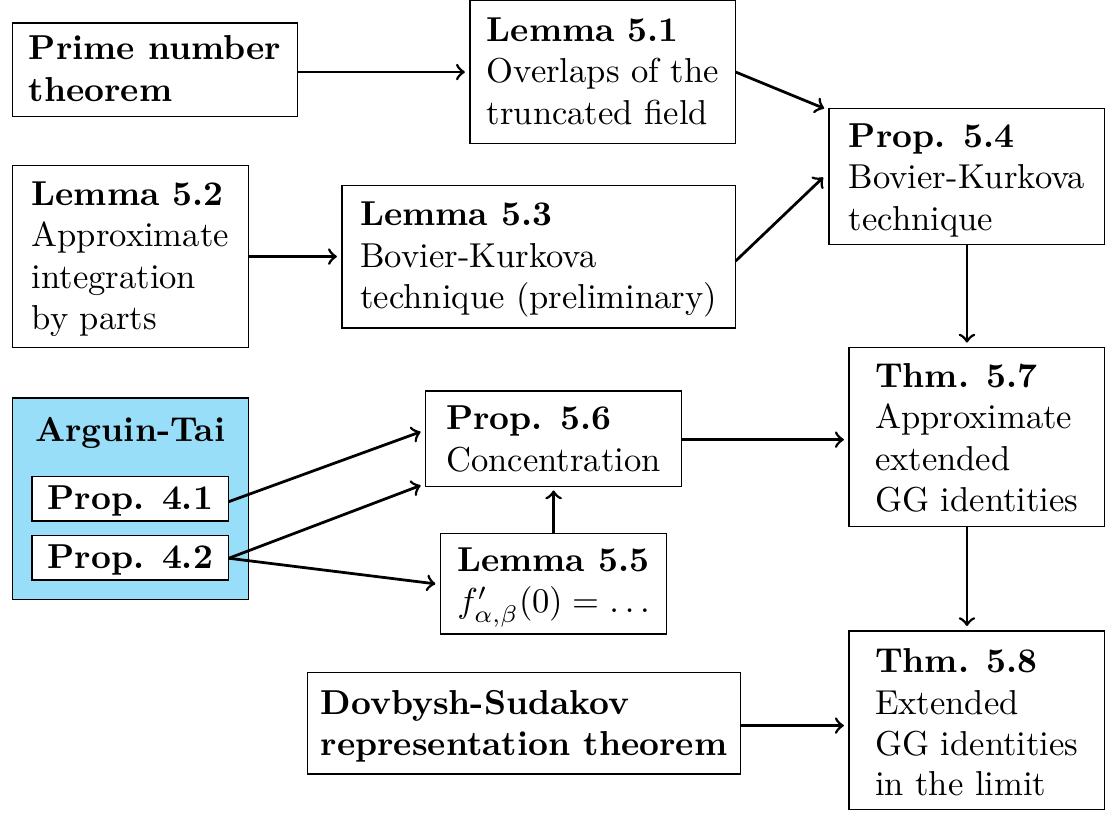}
        \caption{Structure of the proof}
        \label{fig:proof.structure}
    \end{figure}


    \newpage
    We start by relating the overlaps of the field $X$ to the overlaps of the truncated field $X(\alpha)$.

    \begin{lemma}[Overlaps of the truncated field]\label{lem:covariance.estimates}
        Let $0 \leq \alpha \leq 1$. Then, for all $h,h'\in [0,1]$,
        \begin{equation}
            \frac{\EE[X_h(\alpha) X_{h'}(\alpha)]}{\frac{1}{2} \log \log T}
            =
            \left\{\hspace{-1mm}
            \begin{array}{ll}
                \rho(h,h') + O\hspace{-0.5mm}\left((\log \log T)^{-1}\right), &\mbox{if } \rho(h,h') \leq \alpha, \\[1mm]
                \alpha + O\hspace{-0.5mm}\left((\log \log T)^{-1}\right), &\mbox{if } \rho(h,h') > \alpha.
            \end{array}
            \right.
        \end{equation}
        In both cases, the $O\hspace{-0.5mm}\left((\log \log T)^{-1}\right)$ term is uniform in $\alpha$.
    \end{lemma}

    \begin{proof}
        Since $\text{Re}(z) = (z + \overline{z})/2$, $\EE[U_p^2] = \EE[(\overline{U_p})^2] = 0$ and $\EE[U_p \overline{U_p}] = 1$, it is easily shown from \eqref{def:X} that, for any prime $p$,
        \begin{equation}\label{eq:W.covariance}
            \EE[W_p(h) W_p(h')] = \frac{1}{2p} \cos(|h - h'| \log p), \quad h,h'\in [0,1].
        \end{equation}
        Thus, from the independence of the $U_p$'s and \eqref{def:X.alpha},
        \begin{equation}\label{eq:X.alpha.covariance}
            \EE[X_h(\alpha) X_{h'}(\alpha)] = \sum_{p \leq \exp((\log T)^{\alpha})} \frac{1}{2p} \cos(|h - h'| \log p), \quad h,h'\in [0,1].
        \end{equation}
        Sums like the one on the right-hand side of \eqref{eq:X.alpha.covariance} were estimated on page 20 of Appendix A in \cite{arXiv:1304.0677} by using the prime number theorem.
        In particular,
        \begin{equation}\label{eq:rho.estimate}
            \rho(h,h') = \frac{\frac{1}{2} \log\big((\log T) \wedge |h - h'|^{-1}\big)}{\frac{1}{2} \log \log T} + O\hspace{-0.5mm}\left((\log \log T)^{-1}\right),\vspace{-3mm}
        \end{equation}
        and
        \vspace{2mm}
        \begin{equation}\label{eq:rho.estimate.alpha}
            \frac{\EE[X_h(\alpha) X_{h'}(\alpha)]}{\frac{1}{2} \log \log T}
            =
            \left\{\hspace{-1mm}
            \begin{array}{ll}
                \frac{\log|h - h'|^{-1}}{\log \log T} + O\hspace{-0.5mm}\left((\log \log T)^{-1}\right), &\mbox{if } 1 \leq |h - h'|^{-1} < (\log T)^{\alpha}, \\[1.2mm]
                \alpha + O\hspace{-0.5mm}\left((\log \log T)^{-1}\right), &\mbox{if } |h - h'|^{-1} \geq (\log T)^{\alpha},
            \end{array}
            \right.
        \end{equation}
        where the $O\hspace{-0.5mm}\left((\log \log T)^{-1}\right)$ terms are all uniform in $\alpha$.
        By comparing \eqref{eq:rho.estimate} and \eqref{eq:rho.estimate.alpha}, we get
        \begin{align}
            \frac{\EE[X_h(\alpha) X_{h'}(\alpha)]}{\frac{1}{2} \log \log T}
            &= \left\{\hspace{-1mm}
            \begin{array}{ll}
                \rho(h,h') + O\hspace{-0.5mm}\left((\log \log T)^{-1}\right), &\mbox{if } \rho(h,h') - O\hspace{-0.5mm}\left((\log \log T)^{-1}\right) < \alpha, \\[1.2mm]
                \alpha + O\hspace{-0.5mm}\left((\log \log T)^{-1}\right), &\mbox{if } \rho(h,h') - O\hspace{-0.5mm}\left((\log \log T)^{-1}\right) \geq \alpha,
            \end{array}
            \right. \notag \\[2mm]
            &=
            \left\{\hspace{-1mm}
            \begin{array}{ll}
                \rho(h,h') + O\hspace{-0.5mm}\left((\log \log T)^{-1}\right), &\mbox{if } \rho(h,h') \leq \alpha, \\[1.2mm]
                \alpha + O\hspace{-0.5mm}\left((\log \log T)^{-1}\right), &\mbox{if } \rho(h,h') > \alpha.
            \end{array}
            \right.
        \end{align}
        This ends the proof.
    \end{proof}

    The next lemma is an approximate integration by parts result.
    It is a straightforward generalization of Lemma 9 in \cite{arXiv:1706.08462}.

    \begin{lemma}[Approximate integration by parts]\label{lem:approximate.integration.by.parts}
        Let $s\in \N^*$ and let $\bb{\xi} \circeq (\xi_1,\xi_2,\ldots,\xi_s)$ be a random vector taking values in $\C^s$, such that $\EE[|\xi_j|^3] < \infty$ and $\EE[\xi_j] = 0$ for all $j\in \{1,\ldots,s\}$, and such that $\EE[\xi_l \xi_j] = 0$ for all $l,j\in \{1,\ldots,s\}$.
        Let $F : \C^s \to \C$ be a twice continuously differentiable function such that, for some $M > 0$,
        \begin{equation*}
            \max_{1 \leq j \leq s} \left\{\|\partial_{z_j}^2 F\|_{\infty} \vee \|\partial_{\overline{z_j}}^2 F\|_{\infty}\right\} \leq M,
        \end{equation*}
        where $\|f\|_{\infty} \circeq \sup_{\bb{z}\in \C^s} |f(\bb{z},\bb{\overline{z}})|$.
        Then, for any $k\in \{1,\ldots,s\}$,
        \begin{align}
            &\Big|\EE[\xi_k F(\bb{\xi},\overline{\bb{\xi}})] - \sum_{j=1}^s \EE[\xi_k \overline{\xi_j}] ~ \EE[\partial_{\overline{z_j}} F(\bb{\xi},\overline{\bb{\xi}})]\Big| \ll s^2 M \max_{1 \leq j \leq s} \EE[|\xi_j|^3], \label{eq:lem:approximate.integration.by.parts.eq.1} \\
            &\Big|\EE[\overline{\xi_k} F(\bb{\xi},\overline{\bb{\xi}})] - \sum_{j=1}^s \EE[\overline{\xi_k} \xi_j] ~ \EE[\partial_{z_j} F(\bb{\xi},\overline{\bb{\xi}})]\Big| \ll s^2 M \max_{1 \leq j \leq s} \EE[|\xi_j|^3], \label{eq:lem:approximate.integration.by.parts.eq.2}
        \end{align}
        where $f(\cdot) \ll g(\cdot)$ means that $|f(\cdot)| \leq C g(\cdot)$ for some universal constant $C > 0$ (the Vinogradov notation).
    \end{lemma}

    \begin{proof}
        Fix $k\in \{1,\ldots,s\}$.
        We only prove \eqref{eq:lem:approximate.integration.by.parts.eq.1} because the proof of \eqref{eq:lem:approximate.integration.by.parts.eq.2} is almost identical.
        Since $\EE[\xi_k] = 0$ and $\EE[\xi_k \xi_j] = 0$ for all $j\in \{1,\ldots,s\}$, the left-hand side of \eqref{eq:lem:approximate.integration.by.parts.eq.1} can be written as
        \begin{equation}\label{eq:lem:approximate.integration.by.parts.first}
            \begin{aligned}
            &\EE\Big[\xi_k \Big(F(\bb{\xi},\overline{\bb{\xi}}) - F(\bb{0},\bb{0}) - \sum_{j=1}^s \xi_j \partial_{z_j} F(\bb{0},\bb{0}) - \sum_{j=1}^s \overline{\xi_j} \partial_{\overline{z_j}} F(\bb{0},\bb{0})\Big)\Big] \\
            &\hspace{50mm}- \sum_{j=1}^s \EE\Big[\xi_k \overline{\xi_j}\Big] \EE \Big[\partial_{\overline{z_j}} F(\bb{\xi},\overline{\bb{\xi}}) - \partial_{\overline{z_j}} F(\bb{0},\bb{0})\Big].
            \end{aligned}
        \end{equation}
        By Taylor's theorem in several variables and the assumptions, the following estimates hold
        \begin{align}
            &\Big|F(\bb{\xi},\overline{\bb{\xi}}) - F(\bb{0},\bb{0}) - \sum_{j=1}^s \xi_j \partial_{z_j} F(\bb{0},\bb{0}) - \sum_{j=1}^s \overline{\xi_j} \partial_{\overline{z_j}} F(\bb{0},\bb{0})\Big| \notag \\
            &\hspace{50mm}\ll M \left(\sum_{l=1}^s |\xi_l|\right)^2 \leq M \, s \sum_{l=1}^s |\xi_l|^2, \\
            &\Big|\partial_{\overline{z_j}} F(\bb{\xi},\overline{\bb{\xi}}) - \partial_{\overline{z_j}} F(\bb{0},\bb{0})\Big| \ll M \sum_{l=1}^s |\xi_l| \quad \text{for all } j\in \{1,\ldots,s\}.
        \end{align}
        Therefore,
        \vspace{-2mm}
        \begin{align}
            |\eqref{eq:lem:approximate.integration.by.parts.first}|
            &\ll M \sum_{l=1}^s \Big(s \, \EE\big[|\xi_k| \cdot |\xi_l|^2\big] + \sum_{j=1}^s \EE\big[|\xi_k| \cdot |\xi_j|\big] \EE\big[|\xi_l|\big]\Big) \notag \\
            &\leq M \sum_{l=1}^s \Big(s \, \EE\big[|\xi_k|^3\big]^{1/3} \EE\big[(|\xi_l|^2)^{3/2}\big]^{2/3} + \sum_{j=1}^s \EE\big[|\xi_k|^3\big]^{1/3} \EE\big[|\xi_j|^3\big]^{1/3} \EE\big[|\xi_l|^3\big]^{1/3}\Big) \notag \\
            &\leq 2 s^2 M \max_{1 \leq j \leq s} \EE[|\xi_j|^3],
        \end{align}
        where we used Holder's inequality to obtain the second inequality.
    \end{proof}

    Here is a generalization of Proposition 10 in \cite{arXiv:1706.08462}.
    It could be seen as a generalization of \eqref{eq:prop.3.Arguin.Tai.2017.rewrite} if \eqref{eq:prop.3.Arguin.Tai.2017.rewrite} was applied to $(W_p(h), h\in [0,1])$ instead of $(X_h(\alpha), h\in [0,1])$.

    \begin{lemma}[Bovier-Kurkova technique - preliminary version]\label{lem:bovier.kurkova.technique.p}
        Let $\beta > 0$ and $p \leq T$.
        For any $s\in \N^*$, any $k\in \{1,\ldots,s\}$, and any bounded mesurable function $\phi : [0,1]^s \rightarrow \R$, we have
        \begin{equation}\label{eq:prop:bovier.kurkova.technique.p.eq}
            \begin{aligned}
            &\left|\EE G_{\beta,T}^{\times s}[W_p(h_k) \phi(\bb{h})] \right. \\
                &\hspace{10mm}- \left. \beta \cdot \left\{\hspace{-1.5mm}
                    \begin{array}{l}
                        \sum_{l=1}^s \EE G_{\beta,T}^{\times s} \big[\EE[W_p(h_k) W_p(h_l)] \, \phi(\bb{h})\big] \\[1mm]
                        - s \, \EE G_{\beta,T}^{\times (s+1)} \big[\EE[W_p(h_k) W_p(h_{s+1})] \, \phi(\bb{h})\big]
                    \end{array}
                    \hspace{-1.5mm}\right\}
            \right|
            \leq K p^{-3/2},
            \end{aligned}
        \end{equation}
        where $\bb{h} \circeq (h_1,h_2,\ldots,h_s)$, $K \circeq s^2 C \beta^2 \|\phi\|_{\infty}$, and $C > 0$ is a universal constant.
    \end{lemma}

    \begin{proof}
        Write for short
        \begin{equation}\label{eq:def.Omega.p}
            \omega_p(h) \circeq \frac{1}{2} p^{-i h - 1/2} \quad \text{and} \quad Y_p(h) \circeq \beta \sum_{\substack{q \leq T \\ q \neq p}} W_q(h).
        \end{equation}
        \vspace{-2mm}
        \hspace{-2.3mm}
        Define
        \begin{equation}
            F_p(\bb{z},\overline{\bb{z}}) \circeq \frac{\int_{[0,1]^s} \omega_p(h_k) \phi(\bb{h}) \prod_{l=1}^s \exp\Big(\beta (z_l \omega_p(h_l) + \overline{z_l} \overline{\omega_p(h_l)}) + Y_p(h_l)\Big) d \bb{h}}{\int_{[0,1]^s} \prod_{l=1}^s \exp\Big(\beta (z_l \omega_p(h_l) + \overline{z_l} \overline{\omega_p(h_l)}) + Y_p(h_l)\Big) d \bb{h}}.
        \end{equation}
        Then,
        \begin{equation}\label{eq:prop:bovier.kurkova.technique.p.sum.expectations}
            \EE G_{\beta,T}^{\times s}[W_p(h_k) \phi(\bb{h})] = \EE[U_p \cdot F_p(\bb{U}_p,\overline{\bb{U}_p})] + \EE[\overline{U_p} \cdot \overline{F_p}(\bb{U}_p,\overline{\bb{U}_p})],
            \vspace{2mm}
        \end{equation}
        where $\bb{U}_p \circeq (U_p,U_p,\ldots,U_p)$.
        Since the $U_p$'s are i.i.d.\hspace{-0.3mm} uniform random variables on the unit circle in $\C$, we have $\EE[|U_p|^3] < \infty$, $\EE[U_p \overline{U_p}] = 1$ and $\EE[U_p^2] = \EE[U_p] = 0$.
        If we apply \eqref{eq:lem:approximate.integration.by.parts.eq.1} with $F = F_p$ and $\bb{\xi} = \bb{U}_p$, and \eqref{eq:lem:approximate.integration.by.parts.eq.2} with $F = \overline{F_p}$ and $\bb{\xi} = \bb{U}_p$, we get, as $T\to\infty$,
        \begin{equation}\label{eq:prop:bovier.kurkova.technique.p.sum.expectations.next}
            \begin{aligned}
                \EE G_{\beta,T}^{\times s}[W_p(h_k) \phi(\bb{h})]
                &= \sum_{j=1}^s \Big\{\EE\big[\partial_{\overline{z_j}} F_p(\bb{U}_p,\overline{\bb{U}_p})\big] + \EE\big[\partial_{z_j} \overline{F_p}(\bb{U}_p,\overline{\bb{U}_p})\big]\Big\} \\
                &+ s^2 \, O\Big(\max_{1 \leq j \leq s} \big\{\|\partial_{z_j}^2 F_p\|_{\infty} \vee \|\partial_{\overline{z_j}}^2 F_p\|_{\infty}\big\}\Big).
            \end{aligned}
        \end{equation}
        For any bounded mesurable function $H : [0,1] \to \C$, define
        \begin{equation}
            \langle H \rangle_{(z,\overline{z})} \circeq \langle H(h) \rangle_{(z,\overline{z})} \circeq \frac{\int_{[0,1]} H(h) \exp\Big(\beta (z \omega_p(h) + \overline{z} \overline{\omega_p(h)}) + Y_p(h)\Big) d h}{\int_{[0,1]} \exp\Big(\beta (z \omega_p(h) + \overline{z} \overline{\omega_p(h)}) + Y_p(h)\Big) d h},
        \end{equation}
        and for any bounded mesurable function $H : [0,1]^s \to \C$, define
        \begin{equation}
            \langle H \rangle_{(\bb{z},\overline{\bb{z}})}^{\phi} \circeq \langle H(\bb{h}) \rangle_{(\bb{z},\overline{\bb{z}})}^{\phi} \circeq \frac{\int_{[0,1]^s} H(\bb{h}) \phi(\bb{h}) \prod_{l=1}^s \exp\Big(\beta (z_l \omega_p(h_l) + \overline{z_l} \overline{\omega_p(h_l)}) + Y_p(h_l)\Big) d \bb{h}}{\int_{[0,1]^s} \prod_{l=1}^s \exp\Big(\beta (z_l \omega_p(h_l) + \overline{z_l} \overline{\omega_p(h_l)}) + Y_p(h_l)\Big) d \bb{h}}.
        \end{equation}
        Differentiation of the above yields
        \begin{equation}\label{eq:first.derivative.H}
            \begin{aligned}
                &\partial_{\overline{z_j}} \langle H \rangle_{(\bb{z},\overline{\bb{z}})}^{\phi} = \beta \big\{\langle H \overline{\omega(h_j)}\rangle_{(\bb{z},\overline{\bb{z}})}^{\phi} - \langle H \rangle_{(\bb{z},\overline{\bb{z}})}^{\phi} \langle \overline{\omega_p(h_{s+1})} \rangle_{(z_j,\overline{z_j})}\big\}, \\
                &\partial_{z_j} \langle H \rangle_{(\bb{z},\overline{\bb{z}})}^{\phi} = \beta \big\{\langle H \omega(h_j)\rangle_{(\bb{z},\overline{\bb{z}})}^{\phi} - \langle H \rangle_{(\bb{z},\overline{\bb{z}})}^{\phi} \langle \omega_p(h_{s+1}) \rangle_{(z_j,\overline{z_j})}\big\}.
            \end{aligned}
        \end{equation}
        The partial derivatives in \eqref{eq:first.derivative.H} can be used to expand the summands on the right-hand side of \eqref{eq:prop:bovier.kurkova.technique.p.sum.expectations.next}. Indeed, by using the relation $F_p(\bb{z},\overline{\bb{z}}) = \langle \omega_p(h_k) \rangle_{(\bb{z},\overline{\bb{z}})}^{\phi}$ with $\bb{z} = \bb{U}_p$,
        \begin{align}\label{eq:prop:bovier.kurkova.technique.p.sum.expectations.next.2}
            &\EE\big[\partial_{\overline{z_j}} F_p(\bb{U}_p,\overline{\bb{U}_p})\big] + \EE\big[\partial_{z_j} \overline{F_p}(\bb{U}_p,\overline{\bb{U}_p})\big] \notag \\[1mm]
            &\hspace{5mm}\stackrel{\phantom{\eqref{eq:first.derivative.H}}}{=} \EE\Big[\partial_{\overline{z_j}} \langle \omega_p(h_k) \rangle_{(\bb{U}_p,\overline{\bb{U}_p})}^{\phi}\Big] + \EE\Big[\partial_{z_j} \big\langle \overline{\omega_p(h_k)} \big\rangle_{(\bb{U}_p,\overline{\bb{U}_p})}^{\phi}\Big] \notag \\[1mm]
            &\hspace{5mm}\stackrel{\eqref{eq:first.derivative.H}}{=} \beta \, \EE\Big[\big\langle \omega_p(h_k) \overline{\omega_p(h_j)}\big\rangle_{(\bb{U}_p,\overline{\bb{U}_p})}^{\phi} - \langle \omega_p(h_k) \rangle_{(\bb{U}_p,\overline{\bb{U}_p})}^{\phi} \langle \overline{\omega_p(h_{s+1})} \rangle_{(U_p,\overline{U_p})}\Big] \notag \\[1mm]
            &\hspace{5mm}\quad \, \, + \beta \, \EE\Big[\big\langle \overline{\omega_p(h_k) \overline{\omega_p(h_j)}}\big\rangle_{(\bb{U}_p,\overline{\bb{U}_p})}^{\phi} - \langle \overline{\omega_p(h_k) \rangle_{(\bb{U}_p,\overline{\bb{U}_p})}^{\phi} \langle \overline{\omega_p(h_{s+1})}} \rangle_{(U_p,\overline{U_p})}\Big] \notag \\[1mm]
            &\hspace{5mm}\stackrel{\phantom{\eqref{eq:first.derivative.H}}}{=} \beta \cdot \left\{\hspace{-1mm}
                \begin{array}{l}
                    \EE\Big[\big\langle 2 \text{Re}\big(\omega_p(h_k) \overline{\omega_p(h_j)}\big)\big\rangle_{(\bb{U}_p,\overline{\bb{U}_p})}^{\phi}\Big] \\[3mm]
                    - \EE\Big[2 \text{Re}\Big(\langle \omega_p(h_k) \rangle_{(\bb{U}_p,\overline{\bb{U}_p})}^{\phi} \langle \overline{\omega_p(h_{s+1})} \rangle_{(U_p,\overline{U_p})}\Big)\Big]
                \end{array}
                \hspace{-1mm}\right\}.
        \end{align}
        Since, by definition,
        \begin{equation}\label{eq:bracket.U.p.equals.Gibbs}
            \langle \, \cdot \, \rangle_{(\bb{U}_p,\overline{\bb{U}_p})}^{\phi} = G_{\beta,T}^{\times s}[~\cdot~ \phi(\bb{h})],
        \end{equation}
        and
        \begin{align}\label{eq:lem:bovier.kurkova.technique.p.eq.technical.part}
            &2 \text{Re}\Big(\langle \omega_p(h_k) \rangle_{(\bb{U}_p,\overline{\bb{U}_p})}^{\phi} \langle \overline{\omega_p(h_{s+1})} \rangle_{(U_p,\overline{U_p})}\Big) \notag \\
            &\quad= 2 \text{Re}\left(\frac{\int_{[0,1]} \int_{[0,1]^s} \omega_p(h_k) \overline{\omega_p(h_{s+1})} \phi(\bb{h}) \prod_{l=1}^{s+1} \exp\Big(\beta \sum_{p \leq T} W_p(h_l)\Big) d \bb{h} \, d h_{s+1}}{\int_{[0,1]} \int_{[0,1]^s} \prod_{l=1}^{s+1} \exp\Big(\beta \sum_{p \leq T} W_p(h_l)\Big) d \bb{h} \, d h_{s+1}}\right) \notag \\[1mm]
            &\quad= G_{\beta,T}^{\times (s+1)}\big[2 \text{Re}\big(\omega_p(h_k) \overline{\omega_p(h_{s+1})}\big) \phi(\bb{h})\big],
        \end{align}
        and
        \begin{equation}
            2 \, \text{Re} (\omega_p(h) \overline{\omega_p(h')}) = \frac{1}{2p} \cos(|h - h'| \log p) \stackrel{\eqref{eq:W.covariance}}{=} \EE[W_p(h)W_p(h')],
        \end{equation}
        we can rewrite \eqref{eq:prop:bovier.kurkova.technique.p.sum.expectations.next.2} as
        \begin{equation}\label{eq:prop:bovier.kurkova.technique.p.sum.expectations.next.3}
            \begin{aligned}
            &\EE\big[\partial_{\overline{z_j}} F_p(\bb{U}_p,\overline{\bb{U}_p})\big] + \EE\big[\partial_{z_j} \overline{F_p}(\bb{U}_p,\overline{\bb{U}_p})\big] \\[1mm]
            &\hspace{5mm}= \beta \cdot \left\{\hspace{-1mm}
                \begin{array}{l}
                    \EE G_{\beta,T}^{\times s}\big[\EE[W_p(h_k)W_p(h_j)] \phi(\bb{h})\big] \\[1mm]
                    - \EE G_{\beta,T}^{\times (s+1)}\big[\EE[W_p(h_k)W_p(h_{s+1})] \phi(\bb{h})\big]
                \end{array}
                \hspace{-1mm}\right\}.
            \end{aligned}
        \end{equation}
        From \eqref{eq:prop:bovier.kurkova.technique.p.sum.expectations.next} and \eqref{eq:prop:bovier.kurkova.technique.p.sum.expectations.next.3}, we conclude \eqref{eq:prop:bovier.kurkova.technique.p.eq}, as long as, for all $j\in \{1,\ldots,s\}$,
        \begin{equation}
            \|\partial_{z_j}^2 F\|_{\infty} \vee \|\partial_{\overline{z_j}}^2 F\|_{\infty} \leq \widetilde{C} \beta^2 \|\phi\|_{\infty} p^{-3/2},
        \end{equation}
        where $\widetilde{C} > 0$ is a universal constant.
        To verify this last point, note that, by differentiating in \eqref{eq:first.derivative.H},
        \begin{align}\label{eq:second.derivative.z.bound}
            &\partial_{z_j}^2 \langle H \rangle_{(\bb{z},\overline{\bb{z}})}^{\phi}
            = \beta \left\{\hspace{-1mm}
                \begin{array}{l}
                    \partial_{z_j} \langle H \omega(h_j)\rangle_{(\bb{z},\overline{\bb{z}})}^{\phi} - (\partial_{z_j} \langle H \rangle_{(\bb{z},\overline{\bb{z}})}^{\phi}) \langle \omega_p(h_{s+1}) \rangle_{(z_j,\overline{z_j})} \\[1mm]
                    - \langle H \rangle_{(\bb{z},\overline{\bb{z}})}^{\phi} (\partial_{z_j} \langle \omega_p(h_{s+1}) \rangle_{(z_j,\overline{z_j})})
                \end{array}
                \hspace{-1mm}\right\} \notag \\[1mm]
            &\hspace{10mm}= \beta^2 \left\{\hspace{-1mm}
                \begin{array}{l}
                    \langle H \omega^2(h_j)\rangle_{(\bb{z},\overline{\bb{z}})}^{\phi} - \langle H \omega(h_j) \rangle_{(\bb{z},\overline{\bb{z}})}^{\phi} \langle \omega_p(h_{s+1}) \rangle_{(z_j,\overline{z_j})} \\[2mm]
                    - \Big(\langle H \omega(h_j)\rangle_{(\bb{z},\overline{\bb{z}})}^{\phi} - \langle H \rangle_{(\bb{z},\overline{\bb{z}})}^{\phi} \langle \omega_p(h_{s+1}) \rangle_{(z_j,\overline{z_j})}\Big) \langle \omega_p(h_{s+1}) \rangle_{(z_j,\overline{z_j})} \\[1.5mm]
                    - \langle H \rangle_{(\bb{z},\overline{\bb{z}})}^{\phi} \Big(\langle \omega_p^2(h_{s+1})\rangle_{(z_j,\overline{z_j})} - \langle \omega_p(h_{s+1}) \rangle_{(z_j,\overline{z_j})}^2\Big)
                \end{array}
                \hspace{-1mm}\right\}.
        \end{align}
        Using the relation $F_p(\bb{z},\overline{\bb{z}}) = \langle \omega_p(h_k) \rangle_{(\bb{z},\overline{\bb{z}})}^{\phi}$, \eqref{eq:second.derivative.z.bound}, and the triangle inequality, we obtain
        \begin{align}\label{eq:lem:bovier.kurkova.technique.p.end.proof.Jensen}
            |\partial_{z_j}^2 F_p(\bb{z},\overline{\bb{z}})|
            &= \beta^2 \left|\hspace{-1mm}
                \begin{array}{l}
                    \langle \omega_p(h_k) \omega^2(h_j)\rangle_{(\bb{z},\overline{\bb{z}})}^{\phi} \\[2mm]
                    - 2 \langle \omega_p(h_k) \omega(h_j) \rangle_{(\bb{z},\overline{\bb{z}})}^{\phi} \langle \omega_p(h_{s+1}) \rangle_{(z_j,\overline{z_j})} \\[2mm]
                    + 2 \langle \omega_p(h_k) \rangle_{(\bb{z},\overline{\bb{z}})}^{\phi} \langle \omega_p(h_{s+1}) \rangle_{(z_j,\overline{z_j})}^2 \\[2mm]
                    - \langle \omega_p(h_k) \rangle_{(\bb{z},\overline{\bb{z}})}^{\phi} \langle \omega_p^2(h_{s+1})\rangle_{(z_j,\overline{z_j})}
                \end{array}
                \hspace{-1mm}\right| \notag \\[1mm]
            &\leq \beta^2 \left\{\hspace{-1mm}
                \begin{array}{l}
                    \langle |\omega_p(h_k)| \cdot |\omega(h_j)|^2\rangle_{(\bb{z},\overline{\bb{z}})}^{|\phi|} \\[2mm]
                    + 2 \langle |\omega_p(h_k)| \cdot |\omega(h_j)| \rangle_{(\bb{z},\overline{\bb{z}})}^{|\phi|} \langle |\omega_p(h_{s+1})| \rangle_{(z_j,\overline{z_j})} \\[2mm]
                    + 2 \langle |\omega_p(h_k)| \rangle_{(\bb{z},\overline{\bb{z}})}^{|\phi|} \langle |\omega_p(h_{s+1})| \rangle_{(z_j,\overline{z_j})}^2 \\[2mm]
                    + \langle |\omega_p(h_k)| \rangle_{(\bb{z},\overline{\bb{z}})}^{|\phi|} \langle |\omega_p(h_{s+1})|^2\rangle_{(z_j,\overline{z_j})}
                \end{array}
                \hspace{-1mm}\right\}.
        \end{align}
        Since $|\omega_p(h)| = \frac{1}{2} p^{-1/2}$, $\langle 1 \rangle_{(z_j,\overline{z_j})} = 1$ and $\langle 1 \rangle_{(\bb{z},\overline{\bb{z}})}^{|\phi|} \leq \|\phi\|_{\infty}$, we deduce from \eqref{eq:lem:bovier.kurkova.technique.p.end.proof.Jensen} that
        \begin{equation}
            |\partial_{z_j}^2 F_p(z,\overline{z})| \leq \frac{6}{8} \beta^2 \|\phi\|_{\infty} p^{-3/2}.
        \end{equation}
        We obtain the bound on $\|\partial_{\overline{z_j}}^2 F_p\|_{\infty}$ in the same manner.
    \end{proof}

    The next proposition is a consequence of the two previous lemmas. It generalizes \eqref{eq:prop.3.Arguin.Tai.2017.rewrite}, which corresponds to the special case $(k = 1, s = 1, \phi \equiv 1)$.
    The idea for the statement originates from \cite{MR2070334}, and the idea behind the proof generalizes the special-case application in \cite{MR3211001}.
    See \cite{MR3354619,MR3731796} for an application in the context of the Gaussian free field.

    \begin{proposition}[Bovier-Kurkova technique]\label{eq:bovier.kurkova.technique}
        Let $\beta > 0$ and $0 \leq \alpha \leq 1$.
        For any $s\in \N^*$, any $k\in \{1,\ldots,s\}$, and any bounded mesurable function $\phi : [0,1]^s \rightarrow \R$, we have
        \begin{equation}\label{eq:prop:bovier.kurkova.technique.eq}
            \begin{aligned}
            &\left|\frac{1}{\beta} \cdot \frac{\EE G_{\beta,T}^{\times s}\big[X_{h_k}(\alpha) \phi(\bb{h})\big]}{\frac{1}{2} \log \log T} \right.\\[1mm]
            &\hspace{10mm}- \left.\left\{\hspace{-1mm}
                \begin{array}{l}
                    \sum_{l=1}^s \EE G_{\beta,T}^{\times s} \big[\int_0^{\alpha} \bb{1}_{\{y < \rho(h_k,h_l)\}} dy ~ \phi(\bb{h})\big] \\[1mm]
                    - s\, \EE G_{\beta,T}^{\times (s+1)} \big[\int_0^{\alpha} \bb{1}_{\{y < \rho(h_k,h_{s+1})\}} dy ~ \phi(\bb{h})\big]
                \end{array}
                \hspace{-1.5mm}\right\}\right| = O\hspace{-0.5mm}\left((\log \log T)^{-1}\right),
            \end{aligned}
        \end{equation}
        where $\bb{h} \circeq (h_1,h_2,\ldots,h_s)$.
    \end{proposition}

    \begin{proof}
        For any $l\in \{1,\ldots,s+1\}$,
        \begin{equation}\label{eq:prop:BV.technique.beginning}
            \begin{aligned}
                \mathbb{E}G_{\beta,T}^{\times (s+1)} \big[\int_0^{\alpha} \bb{1}_{\{y < \rho(h_k,h_l)\}} dy ~\phi(\bb{h})\big]
                &= \mathbb{E}G_{\beta,T}^{\times (s+1)} \big[\rho(h_k,h_l) \, \bb{1}_{\{\rho(h_k,h_l) \leq \alpha\}}\, \phi(\bb{h})\big] \\
                &+ \mathbb{E}G_{\beta,T}^{\times (s+1)} \big[\alpha \, \bb{1}_{\{\rho(h_k,h_l) > \alpha\}}\, \phi(\bb{h})\big].
            \end{aligned}
        \end{equation}
        On the other hand, if we sum \eqref{eq:prop:bovier.kurkova.technique.p.eq} over the set $\{p ~\text{prime} : p \leq \exp((\log T)^{\alpha})\}$ and divide by $\frac{\beta}{2} \log \log T$, we obtain
        \begin{equation}\label{eq:prop:BV.technique.end}
            \begin{aligned}
            &\left|\frac{1}{\beta} \cdot \frac{\EE G_{\beta,T}^{\times s}[X_{h_k}(\alpha) \phi(\bb{h})]}{\frac{1}{2} \log \log T} \right. \\
                &\hspace{10mm}-
                \left.\left\{\hspace{-1.5mm}
                \begin{array}{l}
                    \sum_{l=1}^s \EE G_{\beta,T}^{\times s} \left[\frac{\EE[X_{h_k}(\alpha)X_{h_l}(\alpha)]}{\frac{1}{2} \log \log T} \, \phi(\bb{h})\right] \\[2mm]
                    - s \, \EE G_{\beta,T}^{\times (s+1)} \left[\frac{\EE[X_{h_k}(\alpha) X_{h_{s+1}}(\alpha)]}{\frac{1}{2} \log \log T} \, \phi(\bb{h})\right]
                \end{array}
                \hspace{-1.5mm}\right\}
            \right| = O\hspace{-0.5mm}\left((\log \log T)^{-1}\right).
            \end{aligned}
        \end{equation}
        Now, one by one, take the difference in absolute value between each of the $s+1$ expectations inside the braces in \eqref{eq:prop:BV.technique.end} and the corresponding expectation on the left-hand side of \eqref{eq:prop:BV.technique.beginning}. We obtain the bound \eqref{eq:prop:bovier.kurkova.technique.eq} by using Lemma \ref{lem:covariance.estimates}.
    \end{proof}

    Our goal now is to combine Proposition \ref{eq:bovier.kurkova.technique} with a concentration result (Proposition \ref{prop:concentration.result}) in order to prove an approximate version of the GG identities (Theorem \ref{thm:approximate.extended.GG.identities}). We will then show that the identities must hold exactly in the limit $T\to \infty$ (Theorem \ref{thm:extended.GG.identities}).
    Before stating and proving the concentration result, we show that $f_{\alpha,\beta}(\cdot)$, the limiting perturbed free energy, is differentiable in an open interval around $0$.

    \begin{lemma}\label{lem:differentiability.limiting.free.energy}
        Let $\beta > \beta_c \circeq 2$ and $0 \leq \alpha \leq 1$. There exists $\delta = \delta(\alpha,\beta) > 0$ small enough that $f_{\alpha,\beta}(\cdot)$ from Proposition \ref{prop:convergence.free.energy} is differentiable on $(-\delta,\delta)$. Also, we have $f_{\alpha,\beta}'(0) = \beta \alpha$.
    \end{lemma}

    \begin{proof}
        Since $\beta > 2$ and $\lim_{u\to 0} V_{\alpha,u} = 1$, there exists $\delta = \delta(\alpha,\beta) > 0$ small enough that, for all $u\in (-\delta,\delta)$,
        \begin{equation}\label{eq:limiting.free.energy.expression}
            f_{\alpha,\beta}(u) =
                \left\{\hspace{-1mm}
                \begin{array}{ll}
                    \beta \sqrt{V_{\alpha,u}} - 1, &\mbox{if } u < 0, \\
                    \beta(\alpha u + 1) - 1, &\mbox{if } u \geq 0.
                \end{array}
                \right.
        \end{equation}
        The differentiability of $f_{\alpha,\beta}(\cdot)$ on $(-\delta,\delta)\backslash \{0\}$ is obvious.
        Also,
        \begin{equation}\label{eq:limiting.free.energy.derivative.linear}
            \frac{f_{\alpha,\beta}(u) - f_{\alpha,\beta}(0)}{u} =
                \left\{\hspace{-1mm}
                \begin{array}{ll}
                    \beta \frac{\sqrt{V_{\alpha,u}} - 1}{u}, &\mbox{if } u < 0, \\
                    \beta\alpha, &\mbox{if } u \geq 0.
                \end{array}
                \right.
        \end{equation}
        Take both the left and right limits at $0$ to conclude.
    \end{proof}

    Here is the concentration result.
    It is analogous to Theorem 3.8 in \cite{MR3052333}, which was proved for the mixed $p$-spin model.
    We give the proof for completeness.

    \begin{proposition}[Concentration]\label{prop:concentration.result}
        Let $\beta > \beta_c \circeq 2$ and $0 < \alpha < 1$.
        For any $s\in \N^*$, any $k\in \{1,\ldots,s\}$, and any bounded mesurable function $\phi : [0,1]^s \rightarrow \R$, we have
        \begin{equation}
            \left|\frac{\EE G_{\beta,T}^{\times s}[X_{h_k}(\alpha) \phi(\bb{h})]}{\log \log T} - \frac{\EE G_{\beta,T}[X_{h_k}(\alpha)]}{\log \log T} \EE G_{\beta,T}^{\times s}[\phi(\bb{h})]\right| = o_T(1),
        \end{equation}
        where $\bb{h} \circeq (h_1,h_2,\ldots,h_s)$.
    \end{proposition}

    \begin{proof}
        By applying Jensen's inequality to the expectation $\mathbb{E}G_{\beta,T}^{\times s}[ \, \cdot \, ]$, followed by the triangle inequality,
        \begin{align*}
            &\big|\mathbb{E}G_{\beta,T}^{\times s}[X_{h_k}(\alpha) \phi(\bb{h})] - \mathbb{E}G_{\beta,T}[X_{h_k}(\alpha)] \mathbb{E}G_{\beta,T}^{\times s}[\phi(\bb{h})]\big| \notag \\[0.5mm]
            &\quad\quad\leq \mathbb{E}G_{\beta,T} \big|X_{h_k}(\alpha) - \mathbb{E}G_{\beta,T} [X_{h_k}(\alpha)]\big| \cdot \|\phi\|_{\infty} \notag \\
            &\quad\quad\leq \left\{\hspace{-1mm}
                \begin{array}{l}
                    \mathbb{E}G_{\beta,T}\big|X_{h_k}(\alpha) - G_{\beta,T}[X_{h_k}(\alpha)]\big| \\[1mm]
                    + \mathbb{E}\big|G_{\beta,T}[X_{h_k}(\alpha)] - \mathbb{E}G_{\beta,T}[X_{h_k}(\alpha)]\big|
                \end{array}
                \hspace{-1mm}\right\} \cdot \|\phi\|_{\infty} \notag \\
            &\quad\quad\circeq \big\{(a) + (b)\big\} \cdot \|\phi\|_{\infty}.
        \end{align*}
        Below, we show that $(a)$ and $(b)$ are $o(\log \log T)$ in Step 1 and Step 2, respectively.

        \vspace{3mm}
        \noindent
        {\bf Step 1}. Note that
        \vspace{-2mm}
        \begin{align}\label{eq:lem:concentration.result.start.step.1}
            (a)
            &= \mathbb{E}G_{\beta,T}\Big|\int_0^1 (X_{h_1}(\alpha) - X_{h_2}(\alpha))\frac{e^{\beta X_{h_2}}}{\int_0^1 \hspace{-0.5mm}e^{\beta X_{z_2}} dz_2} dh_2\Big| \notag \\
            &\leq \mathbb{E}G_{\beta,T}^{\times 2}\big|X_{h_1}(\alpha) - X_{h_2}(\alpha)\big|.
        \end{align}
        For $u \geq 0$, we define a perturbed version of the last quantity, where the Gibbs measure $G_{\beta,T,u}$ is now defined with respect to the field $(u X_h(\alpha) + X_h, h\in [0,1])$ :
        \begin{align}
            D_{\alpha,\beta,T}(u)
            &\circeq \mathbb{E}G_{\beta,T,u}^{\times 2} \big|X_{h_1}(\alpha) - X_{h_2}(\alpha)\big|.
        \end{align}
        We can easily verify that
        \begin{equation}\label{eq:lem:concentration.result.derivative.D}
            D_{\alpha,\beta,T}'(y) = \beta ~\mathbb{E}G_{\beta,T,y}^{\times 3}\Big[\big|X_{h_1}(\alpha) - X_{h_2}(\alpha)\big| \cdot \big(X_{h_1}(\alpha) + X_{h_2}(\alpha) - 2X_{h_3}(\alpha)\big)\Big].
        \end{equation}
        If we separate the expectation in \eqref{eq:lem:concentration.result.derivative.D} in two parts and apply the Cauchy-Schwarz inequality to each one of them, followed by an application of the elementary inequality $(c + d)^2 \leq 2c^2 + 2d^2$, we find, for $y \geq 0$,
        \begin{align}\label{eq:lem:IGFF.ghirlanda.guerra.restricted.2.a.bound.derivative}
            \left|D_{\alpha,\beta,T}'(y)\right|
            &\leq \beta \cdot
                \left\{\hspace{-1mm}
                \begin{array}{l}
                    \mathbb{E}G_{\beta,T,y}^{\times 3} \big|X_{h_1}(\alpha) - X_{h_2}(\alpha)\big| \big|X_{h_1}(\alpha) - X_{h_3}(\alpha)\big| \\[1mm]
                    +\, \mathbb{E}G_{\beta,T,y}^{\times 3} \big|X_{h_1}(\alpha) - X_{h_2}(\alpha)\big| \big|X_{h_2}(\alpha) - X_{h_3}(\alpha)\big|
                \end{array}
                \hspace{-1mm}\right\} \notag \\[1mm]
            &\leq \beta \cdot 2\, \mathbb{E}G_{\beta,T,y}^{\times 2} [(X_{h_1}(\alpha) - X_{h_2}(\alpha))^2] \notag \\
            &\leq \beta \cdot 8\, \mathbb{E}G_{\beta,T,y} [\big(X_h(\alpha) - G_{\beta,T,y}[X_h(\alpha)]\big)^2].
        \end{align}
        Note that $\beta^{-2} (\log \log T) f_{\alpha,\beta,T}''(y) = G_{\beta,T,y} [\big(X_h(\alpha) - G_{\beta,T,y}[X_h(\alpha)]\big)^2]$ and apply inequality \eqref{eq:lem:IGFF.ghirlanda.guerra.restricted.2.a.bound.derivative} in the identity $u D_{\alpha,\beta,T}(0) = \int_0^u D_{\alpha,\beta,T}(y) dy - \int_0^u \int_0^x D_{\alpha,\beta,T}'(y) dy dx$. We obtain, for $u > 0$,
        \begin{align}\label{eq:lem:concentration.result.D.0.bound}
            D_{\alpha,\beta,T}(0)
            &\leq \frac{1}{u} \int_0^u D_{\alpha,\beta,T}(y) dy + \int_0^u \left|D_{\alpha,\beta,T}'(y)\right| dy \notag \\
            &\leq 2 \left(\frac{1}{u} \int_0^u \beta^{-2} (\log \log T) \EE [f_{\alpha,\beta,T}''(y)] dy\right)^{1/2} \notag \\
            &+ 8 \beta \int_0^u \beta^{-2} (\log \log T) \EE[f_{\alpha,\beta,T}''(y)] dy.
        \end{align}
        In order to bound $\frac{1}{u}\int_0^u D_{\alpha,\beta,T}(y) dy$, we separated $D_{\alpha,\beta,T}(y)$ in two parts (with the triangle inequality) and we applied the Cauchy-Schwarz inequality to the two resulting expectations $\frac{1}{u} \int_0^u \mathbb{E}G_{\beta,T,y}[\, \cdot\, ]\, dy$.
        Now, on the right-hand side of \eqref{eq:lem:concentration.result.D.0.bound}, use the convexity of $f_{\alpha,\beta,T}(\cdot)$ and the mean convergence of $f_{\alpha,\beta,T}(z), ~z > -1$, from Proposition \ref{prop:convergence.free.energy}. We get, for all $u > 0$ and all $y\in (0,1)$,
        \begin{align}
            \limsup_{T\rightarrow \infty} \frac{(a)}{\log \log T}
            &\stackrel{\eqref{eq:lem:concentration.result.start.step.1}}{\leq} \limsup_{T\rightarrow \infty} \frac{D_{\alpha,\beta,T}(0)}{\log \log T} \notag \\[0.5mm] &\stackrel{\eqref{eq:lem:concentration.result.D.0.bound}}{\leq} \frac{8}{\beta} \cdot \left(\frac{f_{\alpha,\beta}(u + y) - f_{\alpha,\beta}(u)}{y} - \frac{f_{\alpha,\beta}(0) - f_{\alpha,\beta}(-y)}{y}\right).
        \end{align}
        From Lemma \ref{lem:differentiability.limiting.free.energy}, there exists $\delta = \delta(\alpha,\beta) > 0$ such that $f_{\alpha,\beta}(\cdot)$ is differentiable on $(-\delta,\delta)$. Therefore, take $u \rightarrow 0^+$ and then $y \rightarrow 0^+$ in the above equation to conclude Step 1.

        \vspace{3mm}
        \noindent
        {\bf Step 2}. For all $u\in (0,1)$, let
        \begin{equation}
            \begin{aligned}
                \eta_{\alpha,\beta,T}(u)
                &\circeq \big|f_{\alpha,\beta,T}(-u) - \mathbb{E}[f_{\alpha,\beta,T}(-u)]\big| + \big|f_{\alpha,\beta,T}(0) - \mathbb{E}[f_{\alpha,\beta,T}(0)]\big| \\[1mm]
                &\quad+ \big|f_{\alpha,\beta,T}(u) - \mathbb{E}[f_{\alpha,\beta,T}(u)]\big|.
            \end{aligned}
        \end{equation}
        Differentiation of the free energy gives $f_{\alpha,\beta,T}'(0) = \beta (\log \log T)^{-1} G_{\beta,T}[X_{h_k}(\alpha)]$.
        Then, from the convexity of $f_{\alpha,\beta,T}(\cdot)$,
        \vspace{-1mm}
        \begin{align}
            \beta \cdot \frac{(b)}{\log \log T}
            &= \EE\big|f_{\alpha,\beta,T}'(0) - \EE[f_{\alpha,\beta,T}'(0)]\big| \notag \\[1mm]
            &\leq \left|\frac{\mathbb{E}[f_{\alpha,\beta,T}(u)] - \mathbb{E}[f_{\alpha,\beta,T}(0)]}{u} - \mathbb{E}[f_{\alpha,\beta,T}'(0)]\right| \notag \\[1mm]
            &+ \left|\frac{\mathbb{E}[f_{\alpha,\beta,T}(0)] - \mathbb{E}[f_{\alpha,\beta,T}(-u)]}{u} - \mathbb{E}[f_{\alpha,\beta,T}'(0)]\right| + \frac{\mathbb{E}[\eta_{\alpha,\beta,T}(u)]}{u}.
        \end{align}
        Using the $L^1$ convergence of $f_{\alpha,\beta,T}(z), ~z > -1$, from Proposition \ref{prop:convergence.free.energy}, and the mean convergence of $f_{\alpha,\beta,T}'(0)$ from Proposition \ref{prop:mean.convergence.derivative.free.energy} (the limit is $f_{\alpha,\beta}'(0)$ by Lemma \ref{lem:differentiability.limiting.free.energy}, the convexity of $\EE[f_{\alpha,\beta,T}(\cdot)]$ and $f_{\alpha,\beta}(\cdot)$, and by Theorem 25.7 in \cite{MR0274683}), we deduce that for all $u\in (0,1)$,
        \begin{equation*}
            \limsup_{T\rightarrow \infty} \frac{(b)}{\log \log T}
            \leq \frac{1}{\beta} \cdot \left\{\hspace{-1mm}
                \begin{array}{l}
                    \left|\frac{f_{\alpha,\beta}(u) - f_{\alpha,\beta}(0)}{u} - f_{\alpha,\beta}'(0)\right| \\[2mm]
                    + \left|\frac{f_{\alpha,\beta}(0) - f_{\alpha,\beta}(-u)}{u} - f_{\alpha,\beta}'(0)\right|
                \end{array}
                \hspace{-1mm}\right\},
        \end{equation*}
        Take $u\rightarrow 0^+$ in the last equation, the differentiability of $f_{\alpha,\beta}(\cdot)$ at $0$ (from Lemma \ref{lem:differentiability.limiting.free.energy}) concludes Step 2.
    \end{proof}

    \begin{theorem}[Approximate extended Ghirlanda-Guerra identities]\label{thm:approximate.extended.GG.identities}
        Let $\beta > \beta_c \circeq 2$ and $0 < \alpha < 1$.
        For any $s\in \N^*$, any $k\in \{1,\ldots,s\}$, and any bounded mesurable function $\phi : [0,1]^s \rightarrow \R$, we have
        \begin{equation}
            \begin{aligned}
            &\left|
                \EE G_{\beta,T}^{(s+1)} \big[\int_0^{\alpha} \bb{1}_{\{y < \rho(h_k,h_{s+1})\}} dy \, \phi(\bb{h})\big] \right.\\[1mm]
                &\hspace{10mm}- \left.
                \left\{\hspace{-1mm}
                    \begin{array}{l}
                        \frac{1}{s} \EE G_{\beta,T}^{\times 2} \big[\int_0^{\alpha} \bb{1}_{\{y < \rho(h_1,h_2)\}} dy\big] \EE G_{\beta,T}^{\times s}[\phi(\bb{h})] \\[2mm]
                        + \frac{1}{s} \sum_{l \neq k}^s \EE G_{\beta,T}^{\times s} \big[\int_0^{\alpha} \bb{1}_{\{y < \rho(h_k,h_l)\}} dy \, \phi(\bb{h})\big]
                    \end{array}
                \hspace{-1mm}\right\}
            \right| = o_T(1),
            \end{aligned}
        \end{equation}
        where $\bb{h} \circeq (h_1,h_2,\ldots,h_s)$.
    \end{theorem}

    \begin{proof}
        From Proposition \ref{eq:bovier.kurkova.technique}, Proposition \ref{prop:concentration.result} and the triangle inequality, we get
        \begin{equation}\label{thm:approximate.extended.GG.identities.eq.1}
            \begin{aligned}
            &\left|
                \frac{1}{\beta} \cdot \frac{\EE G_{\beta,T}[X_{h_k}(\alpha)]}{\frac{1}{2} \log \log T} \EE G_{\beta,T}^{\times s}[\phi(\bb{h})] \right. \\[1mm]
                &\hspace{10mm}- \left.
                \left\{\hspace{-1mm}
                \begin{array}{l}
                    \sum_{l=1}^s \EE G_{\beta,T}^{\times s} \big[\int_0^{\alpha} \bb{1}_{\{y < \rho(h_k,h_l)\}} dy ~ \phi(\bb{h})\big] \\[1mm]
                    - s \, \EE G_{\beta,T}^{\times (s+1)} \big[\int_0^{\alpha} \bb{1}_{\{y < \rho(h_k,h_{s+1})\}} dy ~ \phi(\bb{h})\big]
                \end{array}
                \hspace{-1.5mm}\right\}
            \right| = o_T(1).
            \end{aligned}
        \end{equation}
        Furthermore, from Proposition \ref{eq:bovier.kurkova.technique} in the special case $(s=1,k=1,\phi \equiv 1)$,
        \begin{equation}\label{thm:approximate.extended.GG.identities.eq.2}
            \begin{aligned}
            &\left|
                \frac{1}{\beta} \cdot \frac{\EE G_{\beta,T}[X_{h_k}(\alpha)]}{\frac{1}{2} \log \log T} \right. \\[1mm]
                &\hspace{10mm}- \left.
                \left\{\hspace{-1mm}
                \begin{array}{l}
                    \EE G_{\beta,T}^{\times s} \big[\int_0^{\alpha} \bb{1}_{\{y < \rho(h_k,h_k)\}} dy\big] \\[1mm]
                    - \EE G_{\beta,T}^{\times (s+1)} \big[\int_0^{\alpha} \bb{1}_{\{y < \rho(h_1,h_2)\}} dy\big]
                \end{array}
                \hspace{-1.5mm}\right\}
            \right| = O\hspace{-0.5mm}\left((\log \log T)^{-1}\right).
            \end{aligned}
        \end{equation}
        By combining \eqref{thm:approximate.extended.GG.identities.eq.1} and \eqref{thm:approximate.extended.GG.identities.eq.2}, we get the conclusion.
    \end{proof}

    By the representation theorem of Dovbysh and Sudakov \cite{MR666087} (for an accessible proof, see \cite{MR2679002}), we can show (see e.g.\hspace{-0.3mm} the reasoning on page 1459 of \cite{MR3211001} or page 101 of \cite{MR3052333}) that there exists a subsequence $\{T_m\}_{m\in \N^*}$ converging to $+\infty$ such that for any $s\in \N^*$ and any continuous function $\phi : [0,1]^{s(s-1)/2} \rightarrow \R$, we have
    \begin{equation}\label{eq:gibbs.measure.limit}
        \lim_{m\to\infty} \EE G_{\beta,T_m}^{\times \infty}\big[\phi((\rho(h_l,h_{l'}))_{1 \leq l,l' \leq s})\big] = E \mu_{\beta}^{\times \infty} \big[\phi((R_{l,l'})_{1 \leq l,l' \leq s})\big],
    \end{equation}
    where $R$ is a random element of some probability space with measure $P$ (and expectation $E$), generated by the random matrix of scalar products
    \begin{equation}\label{eq:matrix.scalar.products.H}
        (R_{l,l'})_{l,l'\in \N^*} = \big((\rho_l,\rho_{l'})_{\mathcal{H}}\big)_{l,l'\in \N^*},
    \end{equation}
    where $(\rho_l)_{l\in \N^*}$ is an i.i.d.\hspace{-0.3mm} sample from some random measure $\mu_{\beta}$ concentrated a.s.\hspace{-0.3mm} on the unit sphere of a separable Hilbert space $\mathcal{H}$.
    In particular, from Theorem \ref{thm:limiting.two.overlap.distribution}, we have
    \begin{equation}\label{eq:thm:limiting.two.overlap.distribution.eq.mu}
        E \mu_{\beta}^{\times 2} \big[\bb{1}_{\{R_{1,2} \in A\}}\big] = \frac{2}{\beta} \bb{1}_A(0) + \left(1 - \frac{2}{\beta}\right) \bb{1}_A(1), \quad A\in \mathcal{B}([0,1]).
    \end{equation}

    Next, we show the consequence of taking the limit \eqref{eq:gibbs.measure.limit} in the statement of Theorem \ref{thm:approximate.extended.GG.identities}.
    Note that a function $\phi : \{0,1\}^{s(s-1)/2} \rightarrow \R$ can always be embedded in a continuous function defined on $[0,1]^{s(s-1)/2}$.
    Here is the main result of this section.

    \begin{theorem}[Extended Ghirlanda-Guerra identities in the limit]\label{thm:extended.GG.identities}
        Let $\beta > \beta_c \circeq 2$ and $0 < \alpha < 1$. Also, let $\mu_{\beta}$ be a subsequential limit of $\{G_{\beta,T}\}_{T\geq 2}$ in the sense of \eqref{eq:gibbs.measure.limit}.
        For any $s\in \N^*$, any $k\in \{1,\ldots,s\}$, and any functions $\psi : \{0,1\} \rightarrow \R$ and $\phi : \{0,1\}^{s(s-1)/2} \rightarrow \R$, we have
        \begin{equation}\label{thm:extended.GG.identities.to.prove}
            \begin{aligned}
                E \mu_{\beta}^{(s+1)} \big[\psi(R_{k,s+1}) \phi((R_{i,i'})_{1\leq i,i' \leq s})\big]
                &= \frac{1}{s} E \mu_{\beta}^{\times 2} \big[\psi(R_{1,2})\big] E \mu_{\beta}^{\times s}\big[\phi((R_{i,i'})_{1\leq i,i' \leq s})\big] \\
                &+ \frac{1}{s} \sum_{l \neq k}^s E \mu_{\beta}^{\times s} \big[\psi(R_{k,l}) \phi((R_{i,i'})_{1\leq i,i' \leq s})\big].
            \end{aligned}
        \end{equation}
    \end{theorem}

    \begin{remark}
        The functions $\psi$ and $\phi$ have $\{0,1\}$ and $\{0,1\}^{s(s-1)/2}$ as their domain, respectively, because $R_{l,l'}\in \{0,1\}$ $E \mu_{\beta}^{\times 2}$-almost-surely by \eqref{eq:thm:limiting.two.overlap.distribution.eq.mu} and the matrix $(R_{l,l'})_{1 \leq l,l' \leq s}$ is symmetric and its diagonal elements are equal to $1$ $E \mu_{\beta}^{\times s}$-almost-surely by \eqref{eq:matrix.scalar.products.H}.
    \end{remark}

    \begin{proof}[Proof of Theorem \ref{thm:extended.GG.identities}]
        From \eqref{eq:gibbs.measure.limit} and Theorem \ref{thm:approximate.extended.GG.identities} (in the particular case where $\phi$ is a function of the overlaps), we deduce
        \begin{equation}
            \begin{aligned}
            &E \mu_{\beta}^{(s+1)} \big[\int_0^{\alpha} \bb{1}_{\{y < R_{k,s+1}\}} dy \, \phi((R_{i,i'})_{1 \leq i,i' \leq s})\big] \\[1mm]
            &\hspace{15mm}= \frac{1}{s} E \mu_{\beta}^{\times 2} \big[\int_0^{\alpha} \bb{1}_{\{y < R_{1,2}\}} dy\big] E \mu_{\beta}^{\times s}\big[\phi((R_{i,i'})_{1 \leq i,i' \leq s})\big] \\
            &\hspace{15mm}+ \frac{1}{s} \sum_{l \neq k}^s E \mu_{\beta}^{\times s} \big[\int_0^{\alpha} \bb{1}_{\{y < R_{k,l}\}} dy \, \phi((R_{i,i'})_{1 \leq i,i' \leq s})\big].
            \end{aligned}
        \end{equation}
        From \eqref{eq:thm:limiting.two.overlap.distribution.eq.mu}, we know that $\bb{1}_{\{y < R_{i,i'}\}}$ is $E \mu_{\beta}^{\times 2}$-a.s.\hspace{-0.3mm} constant in $y$ on $[-1,0)$ and $[0,1)$ respectively. Therefore, for any $x\in \{-1,0\}$,
        \begin{equation}\label{thm:approximate.extended.GG.identities.finish}
            \begin{aligned}
                &E \mu_{\beta}^{(s+1)} \big[\bb{1}_{\{x < R_{k,s+1}\}} \phi((R_{i,i'})_{1 \leq i,i' \leq s})\big] \\[2mm]
                &\hspace{15mm}= \frac{1}{s} E \mu_{\beta}^{\times 2} \big[\bb{1}_{\{x < R_{1,2}\}}\big] E \mu_{\beta}^{\times s}\big[\phi((R_{i,i'})_{1 \leq i,i' \leq s})\big] \\
                &\hspace{15mm}+ \frac{1}{s} \sum_{l \neq k}^s E \mu_{\beta}^{\times s} \big[\bb{1}_{\{x < R_{k,l}\}} \phi((R_{i,i'})_{1 \leq i,i' \leq s})\big].
            \end{aligned}
        \end{equation}
        But, any function $\psi : \{0,1\} \rightarrow \R$ can be written as a linear combination of the indicator functions $\bb{1}_{\{0 < \, \cdot \, \}}$ and $\bb{1}_{\{-1 < \, \cdot \, \}}$, so we get the conclusion by the linearity of \eqref{thm:approximate.extended.GG.identities.finish}.
    \end{proof}

    \section{Proof of Theorem \ref{thm:Poisson.Dirichlet}}\label{sec:proof.min.result}

        Once we have Theorem \ref{thm:limiting.two.overlap.distribution} and the Ghirlanda-Guerra identities from Theorem \ref{thm:extended.GG.identities}, the proof follows exactly the same steps as in the proof of Theorem 1.5 in \cite{MR3211001}.
        We can show that any subsequential limit $\mu_{\beta}$ of $\{G_{\beta,T}\}_{T\geq 2}$ in the sense of \eqref{eq:gibbs.measure.limit} must satisfy
        \begin{equation}\label{eq:1.RPC}
            \mu_{\beta} = \sum_{k\in \N^*} \xi_k \delta_{e_k}, \quad P-a.s.,
        \end{equation}
        where $\delta$ is the Dirac measure, $(e_k)_{k\in \N^*}$ is a sequence of orthonormal vectors in $\mathcal{H}$ and $\xi$ is a Poisson-Dirichlet variable of parameter $\beta_c / \beta$. Since the space of probability measures on $[0,1]^{\N^* \times \N^*}$ (the space of overlap matrices) is a metric space under the weak topology, the limit in \eqref{eq:gibbs.measure.limit} must hold for the original sequence. Then, \eqref{eq:thm:Poisson.Dirichlet.eq} is a direct consequence of \eqref{eq:1.RPC}.

%
%

\bibliographystyle{amsplain}
\bibliography{Ouimet_2018_GG_bib}

\ACKNO{I would like to thank the anonymous referees and my advisor, Louis-Pierre Arguin, for their valuable comments that led to improvements in the presentation of this paper.}

\end{document}